\documentclass{amsart}

\usepackage{bm}
\def\<{\langle}
\def\>{\rangle}
\usepackage{upgreek}

\def\HH{{\mathcal H}}

\def\MM{{\mathcal M}}

\def\bbH{{\mathcal Q}}

\def\TT{{\mathcal T}}

\def\bbR{\bbR}
\def\bbC{\mathbb{C}}

\def\bbH{\mathbb{H}}

\def\1{\mathbf{1}}

\def\bf1{\mathbf{1}}
\def\bfi{\mathbf{i}}
\def\bfj{\mathbf{j}}
\def\bfk{\mathbf{k}}
\def\bbR{\mathbb{R}}
\def\bbC{\mathbb{C}}

\def\bbH{\mathbb{H}}

\theoremstyle{plain}
\newtheorem{thm}{Theorem}[section]

\newtheorem{lem}[thm]{Lemma}
\newtheorem{prop}[thm]{Proposition}
\newtheorem{example}[thm]{Example}

\theoremstyle{definition}
\newtheorem{defn}{Definition}[section]

\theoremstyle{remark}

\numberwithin{equation}{section}

\title{Quaternion Toeplitz matrices and their fundamental properties}
\author{Muhammad Ahsan Khan}
\address{Department of Mathematics, University of Kotli Azad Jammu \& Kashmir, Kotli, Azad Jammu \&
	Kashmir 11100, Pakistan}
\email{ahsankhan388@hotmail.com}
\author{Sohail Khan}
\address{Department of Mathematics, University of Kotli Azad Jammu \& Kashmir, Kotli, Azad Jammu \&
	Kashmir 11100, Pakistan}
\email{sohailkhanak56@gmail.com}

\begin{document}
	\begin{abstract}
Toeplitz matrices are characterized by their constant diagonals, have been extensively studied in various settings, including over real and complex numbers. However, their study over quaternion is quite sparse. In this paper, we investigate the structure and the algebraic properties of quaternion Toeplitz matrices. Most importantly, we established a complete characterization of all normal  Toeplitz matrices having entries commutative quaternions.  
	\end{abstract}
	\maketitle

	\section{Introduction}
	
A kind of square matrix in which every diagonal that descends from left to right is constant is called a Toeplitz matrix. For instance, in a $n\times n$ matrix, if the first row is $(p_0\quad  p_{-1} \cdots p_{1-n})$ and the first column is $\begin{pmatrix}
	p_0\\
	p_{1}\\
	\vdots\\
	p_{n-1}
\end{pmatrix}
$, then the Toeplitz matrix consists of $n^2$ entries and will look like as:
\[T=
\begin{pmatrix}
	p_0& p_{-1}& p_{-2} & \cdots &p_{1-n}\\
	p_1& p_0 & p_{-1} & \cdots &p_{2-n}\\
	p_{2}& p_{1} & p_{0} & \cdots &p_{3-n}\\
	\vdots & \vdots & \vdots &\ddots & \vdots\\
	p_{n-1} & p_{n-2}& p_{n-3} & \cdots &p_{0}
\end{pmatrix}.
\]
These matrices are useful in many areas of mathematics and practical disciplines, like integral equations, time series analysis, and queuing theory, see \cite{Heinig2001, Mackey1999,Williams2020,Grenander1958, Macklin1984, Nikolski2020,Widom1965,Khan2018, Khan2023, Khan2021}.
In addition to that their algebraic and analytical theory is a vital part of modern analysis and algebra and there exists a wide literature concerning Toeplitz matrices having elements,  from the algebra of complex numbers. The most important references concerning these matrices are \cite{Ye2016,Nikolski2020,Widom1965, Shalom1987, brown1964}  and \cite{Khan2018, Khan2023, Khan2021,  Khan20221, Khan2022, yagoub2024algebrastoeplitzmatricesquaternion}. They studied the general form of complex (block) Toeplitz matrices, which also include applications to probability theory, statistics, and image restoration. A comprehensive and excellent overview of the uses of complex Toeplitz matrices in various areas of pure and applied mathematics is also available in \cite{Widom1965}.

Whereas, quaternions are a complicated system of hypercomplex numbers that expand besides the notion of complex numbers. Rather than two components like a complex number $p_0+p_1\bfi$, quaternions have four components $p_0+p_1\bfi+p_2\bfj+p_3\bfk$ where $\bfi,\bfj$ and $\bfk$ are imaginary units adhering to the following product rules:
$$
\bfi^2=\bfj^2=\bfk^2=\bfi\bfj\bfk=-1.
$$

These numbers are important in algebra because they offer a fundamental example of a division algebra over the real numbers that is not commutative. These were initially introduced by Sir William Rowan Hamilton in 1843 \cite{Hamilton1969, Hamilton1953}, extended the idea of complex numbers into a non commutative four-dimensional setting. The study of other kinds of noncommutative rings and algebras, which are now important in complex algebraic fields like representation theory and ring theory, were made possible by the noncommutative structure of quaternions.

The reality that quaternions constitute a division algebra is regarded as one of their most essential properties. This shows that every nonzero quaternion is non-singular concerning the product, which is quite rare in higher dimensional algebras, just as real and complex numbers algebras do. 

Following Hamilton's discovery of quaternions, Segre suggested commutative quaternions to provide the commutative property in products \cite{Segre1892}. It is possible to split down commutative quaternions into two complex variables \cite{Kosal2014, Kosal2015, Catoni2005, Catoni12005}. Like quaternions, the collection of commutative quaternions is $4$-dimensional. However, this collection includes isotropic and zero divisor elements \cite{Kosal2014, Kosal2015}. It has been observed that commutative quaternion matrices lack an appropriate theory. For a complete study on commutative quaternions and their matrices, see \cite{Segre1892, Kosal2014, Kosal2015, Catoni2005, Catoni12005} and further references therein. Using complex representations of commutative quaternion matrices, the authors of \cite{Kosal2014, Kosal2015} examined various algebraic features of commutative quaternion matrices. We refer the reader to the lectures of \cite{Zhang1997, kleyn2014,Visick2000, Moenck1977OnCC, Altun2021, Gongopadhyay2012}, for detailed study of quaternions and their matrices. Several writers have provided studies on quaternion matrix norms; for example, \cite{Visick2000, Moenck1977OnCC, Mathias1990} and more references referenced therein. 

Classical Toeplitz matrices are generalized to Toeplitz matrices with quaternion entries, in which every entry is a quaternion instead of a commutative scalar. The non commutative aspect of quaternion product creates intrinsic difficulties and applies when the entries are quaternions, affecting the matrix's structure and properties. It would be reasonable to mention here that quaternion Toeplitz matrices have been studied very little from an algebraic point of view. The authors of \cite{yagoub2024algebrastoeplitzmatricesquaternion} have extracted some maximal left algebras and certain algebraic properties as well, but the theory is not developed to the extent that the theory of complex Toeplitz matrices is. The main task of the current paper is to obtain basic algebraic results and to generalize some of the main results of \cite{Gu2003, Shalom1987}, concerning Toeplitz matrices over quaternions and commutative quaternions.
The plan of the paper is as : After the introductory section, we will study quaternions, commutative quaternions, their matrices and the basic properties concerning them. In the third section, we will introduce quaternion Toeplitz matrices and prove several fundamental results related to them. The final chapter is particularly important, as it deals with the complete classification of normal quaternion Toeplitz matrices whose entries are commutative quaternions.

\section{Quaternions, Commutative Quaternions, and Their matrices: Basic Properties}

This section deals with the introduction of quaternions and their arithmetic: product, conjugate, norm, etc. Additionally, we offer quaternion representations as complex $2\times 2$ matrices and real $4\times 4$ matrices. We begin with the following basic definitions. 
\subsection{Quaternions and Their Matrices}
Set a structured basis multiplication in $\bbH$ using the following formulas: $\{\bf1,\bfi,\bfj,\bfk\}$ in a four-dimensional real vector space $\bbH$ (one can pick $\bbH=\bbR^4$, the vector space of rows or columns composed of four real components).

$$
\bf1\bfi=\bfi\bf1=\bfi,\quad  \bf1\bfj=\bfj\bf1=\bfj, \quad \bf1\bfk=\bfk\bf1=\bfk,
$$

$$
\bfi^2=\bfj^2=\bfk^2=ijk=-1, \quad \bfi\bfj=-\bfj\bfi=\bfk,\quad \bfj\bfk=-\bfk\bfj=\bfi,
\quad 
\bfk\bfi =-\bfi\bfk=\bfj,
$$ and by the condition that the multiplication of $\bbH$'s elements commutes with scalar multiplication and is distributive with regard to addition:
$$p(q+r) =pq+pr, \quad  (q+r)p=qp+rp,\quad  p(\alpha q)=\alpha(pq)$$ for each $p,q,r$ in $\bbH$ and $\alpha\in\bbR.$

Notably, the product of any two basis vectors in $\bbH$ is $\pm$ another basis vector, the set $\{\pm\bf1, \pm\bfi, \pm\bfj, \pm\bfk\}$, forms a non-abelian group under multiplication, known as the quaternion group and typically represented by the symbol $Q_8$.
The (real) quaternions are the elements of $\bbH$ which possess the algebraic operations of $\bbH$ as a real vector space, along with multiplication defined as previously described. Obviously the multiplication in $\bbH$ is non commutative.
\begin{prop}\cite{rodman}
	For any $p,q,r\in\bbH$, $\bbH$ is an algebra with the identity $\bf1$: $p(qr)=(pq)r,$ $\bf1 p=p\bf1=p$
	In future work, we use the quaternion $\alpha\bf1;$ to determine the real number $\alpha$. When it is convenient, we identify $\bbC$ with the subalgebra of $\bbH$ spanned by $\bf1$ and $\bfi.$ Since it is evident that the $span\{\bf1, \bfi\}$ over $\bbR$ is isomorphic to $\bbC$.
\end{prop}

\begin{defn}\cite{rodman}
	If $p=p_0+p_1\bfi+p_2\bfj+p_3\bfk,$ We define $\Re(p)=p_0$, the real of $p$. The vector portion (or imaginary part) of $p$ is $\Im(p)=p_1\bfi+p_2\bfj+p_3\bfk,$ $p_0-p_1\bfi-p_2\bfj-p_3\bfk=\Re(p)-\Im(p)$ defines the conjugate of $p$, which is represented by the symbol $\overline{p}$. $\|p\|=\sqrt{\overline{p} p}=\sqrt{p_0^2+p_1^2+p_2^2+p_3^2}\in\bbR$ is the norm of $p$. If $\|p\|=1$, then $p\in\bbH$ is a unit quaternion.
\end{defn}
The following conclusion gathers some of the main properties of quaternions, related to norm, conjugate, and inverse. These properties be crucial in understanding the algebraic structure and geometric interpretation of quaternions.
\begin{prop}\label{basicquaternion}\cite{Zhang1997, rodman}
	Suppose that $p,q\in\bbH.$ Then:
	\begin{itemize}
		\item [(i)] $\overline{p}p=p\overline{p}$; \\
		\item [(ii)]  Each $p\in\bbH\backslash\{0\}$  has an inverse. More specifically, $p(\overline{p}/\|p\|^2)=1;$\\ $p^{-1}=\overline{p}/\|p\|^2\in\bbH;$
		\item [(iii)]$\|p\|=\|\overline{p}\|;$\\ 
		\item [(iv)] In fact, $\|.\|$ is a norm on $\bbH$; to be more specific, for any $p,q\in\bbH$ one has : \\
		$\|p\|\geq 0$ with equality, provided that
		if $p=0;$ $\|p+q\|\leq\|p\|+\|q\|,\quad  \|pq\|=\|qp\|=\|p\|\|q\|$;\\
		
		\item [(v)] $\bfj \alpha \overline{\bfj}=\bfk\alpha \overline{\bfk}=\overline{\alpha}$ for every $\alpha\in\bbC$;\\
		\item [(vi)] $\overline{pq}=\bar{q}\bar{p}$;\\
		\item [(vii)] $p=\overline{p}\iff p\in\bbR;$\\
		\item [(viii)] if $p\in\bbH,$ then $pq=qp$ for every  $q\in\bbH \iff p\in\bbR;$\\
		\item [(ix)] $p$ and $\overline{p}$ are solutions of the following qudratic equations with real coefficients ; $t^2-2\Re(p)t+\|p\|^2=0;$\\
		
		\item [(x)] The inequality of the Cauchy-Schwarz type is $\max\Big\{|\Re(pq)|,|\Im(pq)|\Big\}\leq\|p\|\|q\|;$
	\end{itemize}
\end{prop}

As a result, $\bbH$ is a 4-dimensional division algebra over $\bbR$.
Quaternions are frequently represented as   matrices of size $4$ over $\bbC$. This section will explore the structural properties of these representation.  These representations are described as.

Writing $p=p_0+p_1\bfi+p_2\bfj+p_3\bfk=p_0+p_1\bfi+(p_2+p_3\bfi)\bfj=\tilde{p}+\tilde{q}\bfj$,  with $p_0,p_1,p_2,p_3\in\bbR,$ define 
$$\chi:\bbH\longrightarrow\MM_2[\bbC], \quad\chi(p)=\begin{pmatrix}
	p_0+p_1\bfi & p_2+p_3\bfi \\
	-p_2+p_3\bfi & p_0-p_1\bfi
\end{pmatrix},$$
where, $p_0,p_1,p_2, p_3\in\bbR$, and $\MM_2[\bbC]$ is the $4$ dimensional algebra of matrices over $\bbC$.
\begin{prop}\cite{Zhang1997}
	$\chi$ is a unital isomorphism of $\bbH$ onto the $4$ dimensional algebra of  matrices of the type $\begin{pmatrix} \tilde{p} & \tilde{q} \\-\overline{\tilde{q}} & \overline{\tilde{p}} \end{pmatrix},$ where $\tilde{p}, \tilde{q}\in\bbC$.
\end{prop}
Just like quaternion numbers, one can also express quaternion matrices in a complex combination of complex matrices. Let $M$ be any matrix $n\times n$ with quaternion entries, then we may express $M$ as $M= A_0 + A_1\bfi + A_2\bfj + A_3\bfk= A + B\bfj$, where
$A=A_0+A_1\bfi$ , $B=A_2+A_3\bfi$ are $n\times n$ complex matrices. Denote by $\chi_1 (M)$, the complex representation matrix of quaternion matrix $M$ as
\begin{center}
	$\chi_1(M)=\begin{pmatrix}
		A&  B \\
		-\overline{B} & \overline{A}
	\end{pmatrix} .$
\end{center}

\begin{example}
	Let $M=\begin{pmatrix}
		\bfi+\bfj+\bfk&  \bf1+\bfk \\
		\bfi+\bfj & \bfj+\bfk
	\end{pmatrix}$, then $\chi_1(M)=\begin{pmatrix}
		\bfi&  \bf1+\bfi&\bf1&\bfi \\
		-\bf1+\bfi& -\bfi&\bfi&\bf1\\
		\bfi& \bf1&0&\bf1+\bfi\\
		-\bf1& -\bfi&-\bf1+\bfi&0
	\end{pmatrix}.$
\end{example}
\begin{lem}\label{injective of chi}\cite{alpay}
	The map $\chi_1$ is a continuous, injective ring homomorphism. In particular, if $c\in\mathbb{R}$, and  $M,N$ are quaternion matrices of size $n^2$, then
	\begin{enumerate}
		\item[1.] $\chi_1(cM)=c\chi_1(M)$;
		\item[2.] $\chi_1(M + N) = \chi_1(M) + \chi_1(N);$
		\item[3.] $\chi_1(MN) = \chi_1(M)\chi_1(N);$
		\item[4.] $\chi_1(M^*)= \chi_1(M)^*.$
	\end{enumerate}
\end{lem}
The introduction of the tensor product operator of rank $1$ marks the end of this subsection. 
\begin{defn}\label{tensor}\cite{yagoub2024algebrastoeplitzmatricesquaternion}
	The tensor product $p\otimes q$ of $p,q\in \bbH$ is a $n\times n$ matrix defined by $$(p\otimes q)r = <r,q>p \quad\hbox{for all}\quad r\in\bbH^n.$$
\end{defn}

\subsection{Commutative Quaternions and Their Matrices }
The set $$ \HH =\biggl\{p=p_0+p_1\bfi+p_2\bfj+p_3\bfk| p_0, p_1, p_2, p_3 \in\bbR\biggl\} $$ represents the set of commutative quaternions, where the following product rules are satisfied by the basis elements $\bfi,\bfj,\bfk\notin\bbR$:
$$\bfi^2=\bfk^2=\bfi\bfj\bfk= -\bf1, \quad \bfj^2=\bf1,\quad  \bfi\bfj=\bfj\bfi=\bfk,\quad  \bfj\bfk=\bfk\bfj=\bfi, \quad \bfk\bfi=\bfi\bfk=-\bfj.$$

It is evident from the preceding that the product operation in $\HH$ is commutative, meaning that for each $p,q$ in $\HH$, $pq=qp$.
We now present explicitly in terms of components the product of commutative quaternions. Suppose that $p=p_0+p_1\bfi+p_2\bfj+p_3\bfk$ and $q=q_0+q_1\bfi+q_2\bfj+q_3\bfk$ are in $\HH$, then their product is defined following ways,
\begin{align*}
	pq&=p_0q_0-p_1q_1+p_2q_2-p_3q_3+(p_1q_0+p_0q_1+p_3q_2+p_2q_3)\bfi\\
	+&(p_0q_2+p_2q_0-p_1q_3-p_3q_1)\bfj+(p_3q_0+p_0q_3+p_1q_2+p_2q_1)\bfk.
\end{align*}
By using matrix multiplication, the previous relation can be represented as a real matrix of dimension $16$  as 
$$\begin{pmatrix}
	p_0 & -p_1 & p_2 & -p_3 \\
	p_1 & p_0 & p_3 & p_2 \\
	p_2 & -p_3 & p_0 & -p_1 \\
	p_3 & p_2 & p_1 & p_0
\end{pmatrix}.$$
Which is quite helpful in computing quaternions. 

$$
\begin{pmatrix}
	r_0 \\
	r_1 \\
	r_2 \\
	r_3
\end{pmatrix}
=\begin{pmatrix}
	p_0 & -p_1 & p_2 & -p_3 \\
	p_1 & p_0 & p_3 & p_2 \\
	p_2 & -p_3 & p_0 & -p_1 \\
	p_3 & p_2 & p_1 & p_0
\end{pmatrix}
\begin{pmatrix}
	q_0 \\
	q_1 \\
	q_2 \\
	q_3
\end{pmatrix}, \quad\hbox{where} \quad pq=r_0+r_1\bfi +r_2\bfj+r_3\bfk. 
$$
It should be noted that there are some striking differences between $\HH$ and $\bbH$. For any $p=p_0+p_1\bfi+p_2\bfj+p_3\bfk\in\bbH$, there is only one adjoint $\overline{p}=p_0-p_1\bfi-p_2\bfj-p_3\bfk$, while for $p=p_0+p_1\bfi+p_2\bfj+p_3\bfk\in\HH$, there exist three types of conjugate, called the principal conjugate, symbolized by $p^{(1)}$, $p^{(2)}$ and $p^{(3)}$ and are given as: 
\begin{align*}
	p^{(1)}&=p_0-p_1\bfi+p_2\bfj-p_3\bfk\\
	p^{(2)}&=p_0+p_1\bfi-p_2\bfj-p_3\bfk\\  
	p^{(3)}&=p_0-p_1\bfi-p_2\bfj+p_3\bfk.
\end{align*}
The norm $\|p\|$ of $p$  is given by 
$$\|p\| =\sqrt[4]{|pp^{(1)}p^{(2)}p^{(3)}|} =\sqrt[4]{[(p_0+p_2)^2+(p_1+p_3)^2][(p_0-p_2)^2+(p_1-p_3)^2]}\geq0.$$ 
If $p_0+p_2=0$, $p_1+p_3=0$ or $p_0-p_2=0$, $p_1-p_3=0$, then $\|p\|=0$, and if $\|p\|\neq 0$, then $p$ has an inverse with respect to  multiplication given as  $p^{-1}=\frac{pp^{(1)}p^{(2)}p^{(3)}}{\|p\|^4}$.

The following result shows that, similar to quaternions, we can express every commutative quaternion as a $2\times 2$ complex matrix.
\begin{thm}\cite{Kosal2014}
	Each commutative quaternion has a distinct representation as a complex array of size $4$.
\end{thm}
The algebra of matrices, whose elements are all in $\HH$, is symbolized throughout by $\MM_n[\HH]$.
The matrix addition is the ordinary matrix addition and the product is the usual matrix product. For any $T=(p_{ij}))_{i,j=0}^{n-1}$ in $\MM_n[\HH]$, the scalar multiplication is defined as   
$$
qT=Tq=(qp_{ij})_{i,j=0}^{n-1}, \quad q\in\HH. 
$$
The reader can see that if $T, U$ in $\MM_n[\HH]$ and $q, r\in\HH$, then 
\begin{align*}
	&(qT)U=q(TU),\\
	&(Tq)U=T(qU),\\
	&(qr)T=q(rT).
\end{align*}

Furthermore, $\MM_n[\HH]$ is a free module over $\HH$.  If  $T=(p_{ij})_{i,j=0}^{n-1}\in M_n[\HH],$ then there are three different conjugates, known as principal conjugates:
$T^{(1)}=(p_{ij}^{(1)})_{i,j=0}^{n-1}\in\MM_n[\HH]$, $T^{(2)}=(p_{ij}^{(2)})_{i,j=0}^{n-1}\in \MM_n[\HH]$  and $T^{(3)}=(p_{ij}^{(3)})_{i,j=0}^{n-1}\in\MM_n[\HH]$.
$T^t=(p_{ji})_{i,j=0}^{n-1}\in \MM_n[\HH]$ is the transpose of $T;$ $T^{\dagger \kappa} =(T^{(\kappa)})^t\in \MM_n[\HH]$ is the $\kappa^{th}$ adjoint of $T$, where $\kappa=1,2,3$.
By $\kappa^{th}$ conjugate, any $T\in M_n(\HH)$ is considered a normal matrix if $T$ commutes with $T^{\dagger \kappa}$, that is, $TT^{\dagger \kappa}=T^{\dagger \kappa}T$.

The following result from \cite{Kosal2014} gathers a few of the main needful properties of the matrices of $\MM_n[\HH]$.
\begin{thm}
	\begin{itemize}
		Suppose that $T, U\in \MM_n[\HH]$, the below hold
		\item [(i)]   $(T^{(\kappa)})^t=(T^t)^{(\kappa)};$\\
		\item [(ii)]  $(TU)^{\dagger \kappa}=U^{\dagger \kappa}T^{\dagger \kappa}$, where $\kappa=1,2,3$;\\
		\item[(iii)] $(TU)^t=U^tT^t$;\\
		\item[(iv)]  $(TU)^{(\kappa)}=T^{(\kappa)}U^{(\kappa)}$;\\
		\item[(v)] Fo $\kappa, \rho, \varrho=1,2,3$, one has  $(T^{(\kappa)})^{(\rho)}=\begin{cases}
			T^{(\varrho)} \quad \hbox{if}\quad \kappa\neq\rho\neq\varrho\\
			T \quad \hbox{if} \quad \kappa=\rho.
		\end{cases} 
		$
	\end{itemize}
\end{thm}

	\section{Algebraic Properties of Quaternion Toeplitz Matrices}
	In this section, we study in detail about the structure and the intricate properties of quaternion Toeplitz matrices. We extend the idea of Toeplitz matrices to the quaternion setting, which introduces more complexities arising from the non commutative behavior of $\bbH$. We consider it necessary to clarify here that the techniques employed in proving the results of this section are largely inspired by the methods developed in \cite{Gu2003}. However, our approach is not a simple repetition; rather, it represents an adaptation and refinement of those ideas within the framework of our setting.
	 We now present the formal definition of Toeplitz matrix having all entries in $\bbH$.  
	\begin{defn}
		If the elements of a finite square matrix are constant along each negative sloping diagonal, the matrix is known to as a quaternion Toeplitz matrix. Stated differently, the arrangement of the matrix elements makes each row a shifted version of the one before it.
		In more general setting, we introduce $a_{ij} = t_{i-j},$ whereas $0\leq i,j\leq  n-1.$
	\end{defn}
	Thus a finite quaternion Toeplitz matrix $T$ has the following structure:
	$$T=T(p,\psi)+p_0I=\begin{pmatrix}
		p_0 & \overline{\psi_1} & \overline{\psi_2} & \cdots & \overline{\psi_{n-1}} \\
		p_1 & p_0 & \overline{\psi_1} & \cdots & \overline{\psi_{n-2}} \\
		p_2 & p_1 & p_0 & \cdots & \overline{\psi_{n-3}} \\
		\vdots & \vdots & \vdots & \ddots & \vdots \\
		p_{n-1} & p_{n-2} & p_{n-3} & \cdots & p_0
	\end{pmatrix}.$$
	In this case, the term "quaternion" means that entries in the matrix representation given above come from the algebra $\bbH$.
	
	Let $\Gamma$ be the square matrix consisting of zeroes with the exception of $\bf1$'s positioned precisely along the subdiagonal, i.e., the diagonal immediately below the principal diagonal, that is, $\Gamma$ has the form:
	$$\Gamma=\begin{pmatrix}
		0 & 0 & 0 &  \cdots & 0 & \\
		\bf1 & 0 & 0 & \cdots & 0  & \\
		0 & \bf1 & 0 &  \cdots & 0 & \\
		\vdots & \vdots & \ddots & \cdots & \vdots \\
		0 & 0 & 0 &\cdots & 0
	\end{pmatrix},$$ then the following matrix represents its adjoint, which is indicated below as $\Gamma^*:$
	$$\Gamma^*= \begin{pmatrix}
		0 & \bf1 & 0 & \cdots & 0 \\
		0 & 0 & \bf1 & \cdots & 0 \\
		0 & 0 & 0 & \cdots & 0 \\
		\vdots & \vdots & \ddots & \cdots & \bf1 \\
		0 & 0 & 0 & \cdots & 0
	\end{pmatrix}.$$
	It is apparent from the definition of $\Gamma$ that $\Gamma^n=\Gamma^{*n}=0.$ Throughout this chapter, the following notations will be employed:
	\begin{itemize}
		\item $\MM_n[\bbH]$ indicates the space of finite square matrices with all of their elements in $\bbH$;
		\item $\TT_n[\bbH]$ indicates the space of finite square  Toeplitz matrices with all of their elements in $\bbH$.
	\end{itemize}
	\begin{lem}\label{sum lemma}
		Let $T$ be a matrix of size $n^2$, then
		$$T=\sum_{0\leq \ell\leq n-1}\Gamma^\ell(T-\Gamma T\Gamma^*){\Gamma^\ell}^*.$$
	\end{lem}
	\begin{proof}
		The proof follows directly, since
		\begin{align*}
			\sum_{0\leq \ell\leq n-1}\Gamma^\ell(T-\Gamma T\Gamma^*){\Gamma^\ell}^*
			&=\sum_{0\leq \ell\leq n-1}\Gamma^\ell(T-\Gamma T \Gamma^*){\Gamma^\ell}*\\
			&=\sum_{0\leq \ell\leq n-1}(\Gamma^\ell T{\Gamma^\ell}^*-\Gamma^{\ell+1} T \Gamma^{\ell+1*})=T-\Gamma^n T \Gamma^n*=T.
		\end{align*}
		We use that $\Gamma^n=0$.
	\end{proof}
	As a consequence of the above result, studying the much simpler equation $$T-\Gamma T\Gamma^*=0.$$ suffices to determine whether $T$=0.
	
	The following result gives us the classification of all quaternion Toeplitz matrices among all quaternion matrices. 
	\begin{prop}\label{Toeplitz lemma}
		Suppose that $T$ is in $\MM_n[\bbH]$, then the below assertions are proportionate.
		\begin{itemize}
			\item [(i)] $T$ is in $\TT_n[\bbH]$.
			\item [(ii)]  There exist $p,\psi\in \bbH^n$ such that $T-\Gamma T\Gamma^*=p\otimes e_0+e_0\otimes \psi$.
		\end{itemize}
	\end{prop}
	\begin{proof}
		$(ii) \implies (i)$
		Suppose that $T=(p_{ij})_{i,j=0}^{n-1}$ is in $\MM_n[\bbH]$ satisfying
		$$
		T-\Gamma T\Gamma^*=p\otimes e_0+e_0\otimes \psi\quad \hbox{ for some}\quad p,\psi\in\bbH^n.
		$$ 
		Let $p=\begin{pmatrix}
			p_0\\
			p_1\\
			\vdots\\
			p_{n-1}
		\end{pmatrix}$ and $q=\begin{pmatrix}
			\psi_0\\
			\psi_1\\
			\vdots\\
			\psi_{n-1}
		\end{pmatrix}$ be arbitrary vectors in $\bbH^n$.  
		The vectors $e_0, e_1, \cdots e_{n-1}$ form the standard Hamel basis for $\bbH^n$, then we have
		\begin{align*}
			x\otimes e_0(e_j)&=<e_j,e_0>x \\
			&=\begin{cases}
				\displaystyle\sum_{0\leq k\leq n-1}p_0e_k\quad\hbox{if}\quad j=0 \\
				0 \quad\quad\quad\quad\hbox{if}\quad j=1,2,\cdots, n-1.
			\end{cases}
		\end{align*}
		Similarly the action of $e_0\otimes \psi$ on $e_j$ yields that
		\[e_0\otimes \psi(e_j)=<e_j,\psi>e_0=\psi_je_0
		\quad \hbox{for every} \quad j=0, 1,\cdots, n-1.\]
		Thus \begin{equation}\label{basis matrix}
			x\otimes e_0+e_0\otimes y =\begin{pmatrix}
				p_0+\psi_0 & \psi_1 & \cdots & \psi_{n-1} \\
				p_1 & 0 & \cdots & 0 \\
				\vdots & \vdots & \cdots & \vdots \\
				p_{n-1} & 0 & \cdots & 0
			\end{pmatrix}.
		\end{equation}
		On the other hand, simple computation implies that
		\begin{equation}\label{general matrix}T-\Gamma T\Gamma^*=\begin{pmatrix}
				p_{00} & p_{01} & p_{02} & \cdots & p_{0,{n-2}} &  p_{0,{n-1}} \\
				p_{10} & p_{11}-p_{00} & p_{12}-p{01} & \cdots & p_{1,{n-2}}-p{0, n-1} &  p_{1,{n-1}}-p_{0, n-2} \\
				p_{20} & p_{21}-p_{10} & p_{22}-p_{11} & \cdots & p_{2,{n-2}}-p_{1, n-3} &  p_{2,{n-1}}-p_{1, n-2} \\
				\vdots & \vdots & \vdots & \cdots & \vdots & \vdots \\
				p_{{n-1},0} & p_{{n-1},1}-p_{n-2,0} &  p_{{n-1},2}-p_{n-2,1} & \cdots & p_{n-1,{n-2}}-p_{n-2,n-3 } & p_{{n-1},{n-1}}-p_{n-2,n-2}
			\end{pmatrix}.
		\end{equation}
		We have from the comparison of the entries of (\ref{basis matrix}) and (\ref{general matrix}),
		$$
		p_{11}-p_{00}=0, \cdots ,p_{{n-1},{n-1}}=p_{{n-2},{n-2}}
		,$$
		i.e., $T$ is in $\TT_n[\bbH]$. This is the result we aimed to establish.
		
		$(i)\implies (ii)$ Let us assume that $T$ has a Toeplitz structure, that is
		$$T=\begin{pmatrix}
			p_0 & \overline{\psi_1}  & \cdots  & \overline{\psi_{n-1}} \\
			p_1 & p_0  & \cdots &  \overline{\psi_{n-2}} \\
			p_2 & p_1 & \cdots &  \overline{\psi_{n-3}} \\
			\vdots & \vdots  & \ddots & \vdots & \\
			p_{n-1} & p_{n-2}  & \cdots  & p_0
		\end{pmatrix},$$ then $T-\Gamma T\Gamma^*$ is a matrix whose entries all are zero except the entries at the position $(0,j)$ and $(j,0)$ for every $j=0,1,2, \cdots n-1$, i.e.,
		$$T-\Gamma T\Gamma^*=\begin{pmatrix}
			p_0 & \overline{\psi_1} & \overline{\psi_2} & \cdots & \overline{\psi_{n-1}} \\
			p_1 & 0 & 0 & \cdots &  0 \\
			p_2 & 0 & 0 & \cdots &  0 \\
			\vdots & \vdots & \vdots & \cdots & \vdots \\
			p_{n-1} & 0 & 0 & \cdots & 0
		\end{pmatrix}. $$
		The reader can easily verified that the identity $T-\Gamma T\Gamma^*=p\otimes e_0+e_0\otimes \psi$ is valid, if one take
		$p=\begin{pmatrix}
			p_0 \\
			p_1 \\
			p_2 \\
			\vdots \\
			p_{n-1}
		\end{pmatrix},$
		and
		$\psi=\begin{pmatrix}
			0 \\
			{\psi_1} \\
			{\psi_2} \\
			\vdots \\
			{\psi_{n-1}}
		\end{pmatrix}$. The proof is finished.
	\end{proof}
	The aforementioned lemma leads to the pleasant conclusion that follows. 
	\begin{prop} Suppose that  $T\in\MM_n[\mathbb{H}]$, then the below are analogous
		\begin{itemize}
			\item [(i)] $T$ is in $\TT_n[\bbH].$
			\item [(ii)] $\chi_1(T)$ is in $\TT_{2n}[\bbC].$
		\end{itemize}
	\end{prop}
	\begin{proof}
		We prove only $(ii)\implies (i)$ as $(i)\implies (ii)$ is trivial. Let us suppose that $\chi_1(T)$ is in $\TT_n[\bbH]$, then by Lemma \ref{Toeplitz lemma} 
		$$\chi_1(T)-\Gamma_{2n}\chi_1(T)\Gamma_{2n}^* = \chi_1(p)\otimes e_{0}^\prime+ e_{0}^\prime \otimes \chi_1(\psi).$$
		Where $e_0^\prime$ is $2n\times 1$ matrix consisting of $1's$ at the zeroth position and $0$ elsewhere, and $\Gamma_{2n}=\chi_1(\Gamma)$. Then by Lemma \ref{injective of chi}, we have $T-\Gamma T\Gamma^*=p\otimes e_0+e_0\otimes \psi$, i.e., $T$ is in $\TT_n[\bbH]. $
	\end{proof}
	\section{Product of Quaternion Toeplitz Matrices}
	The central technical result of the current section pertaining to the product of quaternion Toeplitz matrices is the subsequent Lemma.
	\begin{lem}\label{technical product lemma}
		Assume that $T=T(p,\psi)+p_0I$ and $U=T(q,\phi)+q_oI$ with $p_0,q_0\in\bbR$, then
		\begin{equation}\label{lemmaeq}TU-\Gamma TU\Gamma^*=p\otimes\phi-\tilde{\psi}\otimes\tilde{q}+[Tq+q_0p+p_0q_0e_0]\otimes e_0+e_0\otimes[\Gamma
			U^*\Gamma^*\psi+\overline{p_0}\phi].
		\end{equation}
	\end{lem}
	\begin{proof}
		Let us denote that $T(p,\psi) $ and $T(q,\phi) $ by $\hat{T}$ and $\hat{U}$ respectively. Then since $p_0, q_0\in \bbR$, then they must commute with every quaternion so,  we can express
		\begin{align}\label{tag}
			TU-\Gamma TU\Gamma^*
			&=[\hat{T}+p_0I][\hat{U}+q_0I]-\Gamma [\hat{T}+p_0I][\hat{U}+q_0I]\Gamma^*\notag\\
			&=\hat{T}\hat{U}+p_0\hat{U}+q_0\hat{T}+p_0q_0I-\Gamma\hat{T}\hat{U}\Gamma^*-\Gamma p_0\hat{U}\Gamma^*-\Gamma q_0\hat{T}\Gamma^*-p_0q_0\Gamma\Gamma^*\notag\\
			&=[\hat{T}\hat{U}- \Gamma\hat{T}\hat{U}\Gamma^*]+p_0[\hat{U}-\Gamma \hat{U}\Gamma^*]+q_0[\hat{T}-\Gamma \hat{T}\Gamma^*]+p_0 q_0[I-\Gamma\Gamma^*].
		\end{align}
		Since $\hat{U}$ and $\hat{T}$ are quaternion Toeplitz matrices, then it follows from Lemma \ref{Toeplitz lemma}, $\hat{U}-\Gamma\hat{U}\Gamma^*=q\otimes e_0+e_0\otimes \phi$, $\hat{T}-\Gamma\hat{T}\Gamma^*=p\otimes e_0+e_0\otimes \psi$. Also, note that $I-\Gamma\Gamma^*=e_0\otimes e_0$. Then the  equality \eqref{tag} above takes the form
		
		\begin{equation}\label{equation 1}
			TU-\Gamma TU\Gamma^*=\hat{T}\hat{U}+p_0[q\otimes e_0+e_0\otimes\phi]+q_0[p\otimes e_0+e_0\otimes\psi]+p_0q_0[e_0\otimes e_0].
		\end{equation}
		Now
		\begin{align}\label{tag1}
			\hat{T}\hat{U}-\Gamma \hat{T}\hat{U}\Gamma^*
			&=\hat{T}\hat{U}-\hat{T}\Gamma\hat{U}\Gamma^*+\hat{T}\Gamma\hat{U}\Gamma^*-\Gamma\hat{T}[\Gamma^*\Gamma+e_{n-1}\otimes e_{n-1}]\hat{U}\Gamma^*\notag\\
			&=\hat{T}[\hat{U}-\Gamma \hat{U}\Gamma^*]+[\hat{T}-\Gamma \hat{T}\Gamma^*]\hat{T}\Gamma\hat{U}\Gamma^*-\Gamma\hat{T}[e_{n-1}\otimes e_{n-1}]\hat{U}\Gamma^*\notag\\
			&=\hat{T}[q\otimes e_0+e_0\otimes\phi]+[p\otimes e_0+e_0\otimes\psi][\Gamma\hat{U}\Gamma^*]-\Gamma\hat{T}e_{n-1}\otimes \Gamma\hat{U}^*e_{n-1}\notag\\
			&=\hat{T}q\otimes e_0+p\otimes\phi+p\otimes \Gamma\hat{U^*}\Gamma^*e_0+e_0\otimes \Gamma\hat{U^*}\Gamma^*\psi-\tilde{\psi}\otimes\tilde{q}\notag\\
			&=\hat{T}q\otimes e_0+p\otimes\phi+e_0\otimes \Gamma\hat{U^*}\Gamma^*\psi-\tilde{\psi}\otimes\tilde{q}.
		\end{align}
		
		By combining \eqref{tag}, \eqref{equation 1}, and \eqref{tag1},  we  have arrived at 
		\begin{align*}
			TU-\Gamma TU\Gamma^*
			&=p\otimes \phi-\tilde{\psi}\otimes\tilde{q}+[\hat{T}q_0+p_0q+q_0p+p_0q_0e_0]\otimes e_0+\\
			&+e_0\otimes[\Gamma\hat{U}^*\Gamma^*\psi+\overline{p_0}\phi+\overline{q_0}\psi]\\
			&=p\otimes\phi-\tilde{\psi}\otimes\tilde{q}+[Tq+q_0p+p_0q_0e_0]\otimes e_0+e_0\otimes[\Gamma
			U^*\Gamma^*\psi+\overline{p_0}\phi].
		\end{align*}
		The proof is therefore complete.
	\end{proof}
	
	The ensuing theorem delineates the conditions that are both required and adequate for the product $TU-VW$ of quaternion Toeplitz matrices $T$, $U$, $V$, and $W$ retain its structure as a quaternion Toeplitz matrix. 
	\begin{thm}\label{producttoeplitz} 
		Suppose that $T=T(p,\psi)+p_0I$, $U=T(q,\phi)+q_0I$, $V=T(r,\lambda)+r_0I$, and $W=T(s,\mu)+s_0I.$ where, $p_0,q_0,r_0,s_0\in \bbR$, then allowing are equivalent
		\begin{itemize}
			\item[(i)] $TU-VW$ is in $\TT_n[\bbH]$.
			\item[(ii)]  $p\otimes\phi-\tilde{\psi}\otimes\tilde{q}=r\otimes\mu-\tilde{\lambda}\otimes\tilde{s}$. 
		\end{itemize}
	\end{thm}
	\begin{proof}
		It is inferred from Lemma (\ref{technical product lemma}) that
		\begin{align}\label{Q2}
			TU-\Gamma TU\Gamma^*-VW+\Gamma VW \Gamma^*
			&\notag =\psi\otimes\phi-\tilde{\psi}\otimes\tilde{q}-r\otimes\mu+\tilde{\lambda}\otimes\tilde{s}+\\
			&\notag+
			[Tq+q_0p+p_0q_0e_0-
			Vs-s_0r-r_0s_0e_0]\otimes e_0\\
			&
			+e_0\otimes[\Gamma U^*\Gamma^*\psi+\overline{p}_0\phi-\Gamma W^*\Gamma^*\lambda-\overline{r}_0\mu].
		\end{align}
		The \eqref{Q2}'s right-hand first four terms clearly relate to vectors that include $0$ in the zeroth position. As a direct implication of Lemma \ref{Toeplitz lemma}, $TU-VW$ is in $\TT_n[\bbH]$ iff
		$$p\otimes\phi-\tilde{\psi}\otimes\tilde{q}-r\otimes\mu+\tilde{\lambda}\otimes\tilde{s}=0.$$
		This is precisely what we wanted to prove.
	\end{proof}
	The ensuing result articulates the conditions that are both required and appropriate under which the  $TU-VW$ vanishes, $T$, $U$, $V$, and $W$ being quaternion Toeplitz matrices. 
	\begin{thm}\label{productzero}
		Consider that $T=T(p,\psi)+p_0I$, $U=T(q,\phi)+q_0I$, $V=T(r,\lambda)+r_0I$, and $W=T(s,\mu)+s_0I$, where, $p_0,q_0,r_0,s_0\in \bbR$. Suppose further that $TU-VW$ is in $\TT_n[\bbH]$, then the below are analogous
		\begin{itemize}
			\item[(i)] $TU-VW$=0.\\
			\item[(ii)]\begin{equation}\label{Q3}Tq+q_0p+p_0q_0e_0=Vs+s_0r+r_0s_0e_0
			\end{equation}
			and
			\begin{equation}\label{Q1}U^*\psi+\overline{p}_0\phi+\overline{p}_0 \overline{q}_0e_0=W^*\lambda+\overline{r}_0\mu+\overline{r}_0\overline{s}_0e_0.
			\end{equation}
		\end{itemize}
	\end{thm}
	\begin{proof}
		Because $TU-VW$ is in $\TT_n[\bbH]$, then $TU-VW=0\iff (TU-VW)-\Gamma (TU-VW) \Gamma^*$=0.
		The preceding equation is valid if and only if the vectors involved in the tensor products with $e_0$ in \eqref{Q2} of the proof of Theorem \ref{producttoeplitz} are $0$. That is $TU=VW$ iff
		$$
		Tq+q_0p +p_0q_0e_0=Vs+s_0r+r_0s_0e_0,
		$$
		\begin{equation}\label{Q4}
			\Gamma U^*\Gamma^*\psi+\overline{p}_0\phi=\Gamma W^*\Gamma^*\lambda+\overline{r}_0\mu.   
		\end{equation}
		Nevertheless, \eqref{Q3} and \eqref{Q1} are equivalent to the preceding two equations, as 
		$$
		\Gamma T\Gamma^* U^*\psi+\overline{p}_0\overline{q}_0e_0=\Gamma T\Gamma^* W^*\lambda+\overline{r}_0\overline{s}_0e_0.
		$$
		is derived from the subtraction of \eqref{Q4} from \eqref{Q1}. 
		As implied by Lemma \ref{Toeplitz lemma}, this is identical as
		$$<\psi,q>e_0+\overline{p}_0\overline{q}_0e_0=<\lambda,s>e_0+\overline{r}_0\overline{s}_0e_0.$$
		Up to complex conjugation, this corresponds to the zeroth component relation of \eqref{Q3}.
	\end{proof}
	The below result elucidates when the product of two non-zero quaternion Toeplitz matrices is zero. 
	\begin{thm}
		Suppose that $T=T(p,\psi)+p_0I$ and  $U=T(q,\phi)+q_0I$ be nonzero. 
		\begin{itemize}
			\item[(i)] If $p$ and $\tilde{\psi}$ are linearly independent, consequently, $TU=0$ suggests that $U=0$.
			
			\item [(ii)] If $\psi=0$, then $TU=0$ implies that either $U=0$ or $p_0=q_0=0$ and $\phi=0$.
			
			In the previously considered case, $TU=0$ if and only if either
			$p=\begin{pmatrix}
				0 \\
				p_1 \\
				\vdots \\
				p_{n-1}
			\end{pmatrix}$ and 
			$q=\begin{pmatrix}
				0 \\
				\vdots \\
				0 \\
				q_{n-1}
			\end{pmatrix}$
			or
			$q=\begin{pmatrix}
				0 \\
				q_1 \\
				\vdots \\
				q_{n-1}
			\end{pmatrix}$
			and
			$p=\begin{pmatrix}
				0 \\
				\vdots \\
				0 \\
				p_{n-1}
			\end{pmatrix}.$
			
			Likewise, if $p=0$, then $TU=0$, and $U\neq 0$ implies that $p_0=q_0=0$ and $q=0.$ In this case, $TU=0\iff$ either
			$\psi=\begin{pmatrix}
				0 \\
				\psi_1 \\
				\vdots \\
				\psi_{n-1}
			\end{pmatrix}$
			and
			$\phi=\begin{pmatrix}
				0 \\
				0 \\
				\vdots \\
				\phi_{n-1}
			\end{pmatrix}$
			or
			$\phi=\begin{pmatrix}
				0 \\
				\phi_1 \\
				\vdots \\
				\phi_{n-1}
			\end{pmatrix}$
			and
			$\psi=\begin{pmatrix}
				0 \\
				0 \\
				\vdots \\
				\psi_{n-1}
			\end{pmatrix}$.
			\item[(iii)] If $p=\lambda{\tilde\psi}$ for some $\lambda\in\bbH$, then $TU=0$ if and only if $q=\lambda{\tilde\phi}$ and 
			\begin{equation}\label{Q5}
				Tq+q_0p+p_0q_0e_0=0.
			\end{equation}
		\end{itemize}
	\end{thm}
	\begin{proof}
		We give the proof of (i) and (ii), as (iii) is quite easy for the reader. 
		By the Theorem \ref{productzero} that if $TU=0$ then, $$p\otimes\phi-\tilde{\psi}\otimes\tilde{q}=0.$$
		
		(i) If $p$ and $\tilde{\psi}$ are linearly independent, then $\phi=\tilde{q}=0.$ Therefore $U=q_0I$. Consequently, $TU=0$ suggests that $U=0.$\\
		(ii) 
		Either $p=0$ or $\phi=0$ if $\psi=0$. Since $p=0$, $T=p_0I$, and $TU=0$ in this instance, $U=0.$ Assume, therefore, that $\phi=0.$
		$T$ and $U$ are both lower triangular, in other words. $T$ or $U$ is invertible if either $p_0$ or $q_0$ is not zero.

		Thus $p_0=q_0=0.$
		$$T(p,0)q=\begin{pmatrix}
			0 & 0 & 0 & \cdots & 0 \\
			p_1 & 0 & 0 & \cdots & 0 \\
			p_2 & p_1 & \ddots & \ddots & \vdots \\
			\vdots & \ddots & \ddots & \ddots & 0 \\
			p_{n-1} & p_{n-2} & \cdots & p_1 & 0
		\end{pmatrix}
		\begin{pmatrix}
			0 \\
			q_1 \\
			q_2 \\
			\vdots \\
			q_{n-1}
		\end{pmatrix}
		=\begin{pmatrix}
			& 0 \\
			& 0 \\
			& p_1q_1 \\
			& \vdots \\
			p_{n-2}q_1+\cdots +& p_1q_{n-2}
		\end{pmatrix}.$$
		The first pair of equations is represented by $T(p,0)=0$ for $p_1\neq 0,$ which implies that $q_i$ vanishes for $1 \leq i\leq n-2$. The second set of equations is represented by $q_i\neq 0$; the proof for $p=0$ is comparable.
	\end{proof}
	\section{Classification of Normal Toeplitz Matrices with Commuting Quaternion Entries}
	Normal matrices occupy a singular place in the space of matrices due to their commutativity with their adjoint. We take the matrices entries from $\HH$ in this section. The classification of normal Toeplitz matrices with elements from $\HH$ will then be presented. 
	
	We begin with the lemma, which can be simply proved by mathematical induction.
	
	\begin{lem}\label{normal}
		If $T=(p_{i,j})_{i,j=0}^{n-1}\in\MM_{n}[\HH]$, then $T$ is normal with respect to $\kappa^{th}$ adjoint if and only if
		\[
		\sum_{0\leq k\leq n-1}[p_{k,\ell}^{(\kappa)}p_{k,\ell}-p_{\ell,k}p_{\ell,k}^{(\kappa)}]=0\quad\hbox{whenever }\quad \ell = 0,1 ,\cdots, n-1,\quad\hbox{where}\quad k\neq \ell\]
		and
		\[\sum_{0\leq k\leq n-1 }[p_{k,i}^{(\kappa)}p_{k,j}-p_{i,k}p_{j,k}^{(\kappa)}]=0 \quad\hbox{whenever}\quad  0\leq i<j\leq n-1.
		\]	
	\end{lem}
	The below is the primary result of this section as it deals with characterizing normal Toeplitz matrices. 
	\begin{thm}\label{normal th100}
		If $T = T(p,\psi)+p_0I\in \TT_n[\HH]$,
		then $T^{\dagger \kappa}T-TT^{\dagger \kappa}=0$ iff for each $k$ and $s$, with $1\leq s,k\leq n-1,$
		\begin{equation}\label{condition}
			p_sp_k^{(\kappa)} +p_{n-s}^{(\kappa)}p_{n-k}=\psi_{s}\psi^{(\kappa)}_{k}+\psi^{(\kappa)}_{n-s}\psi_{n-k}.
		\end{equation}
	\end{thm}
	\begin{proof}
		Suppose that $T^{\dagger \kappa}T-TT^{\dagger \kappa}=0$  and let $U=(u_{i,j})_{i,j=1}^n=T^{\dagger \kappa}T-TT^{\dagger \kappa}$. Since
		$T\in\TT_{n}[\HH]$, from Lemma \ref{normal} that $T^{\dagger \kappa}T-TT^{\dagger \kappa}=0$ precisely when for each  $1\leq r \leq n-1 $ and $1 \leq i < j \leq n ,$
		\begin{equation}\label{normal T}
			u_{r,r}=\sum_{1\leq k\leq r-1}\left[\psi_{k}\psi_k^{(\kappa)}-p_{k}^{(\kappa)}p_{k}\right]-\sum_{1\leq k\leq n-r}\left[\psi_{k}\psi^{(\kappa)}_{k}-p_{k}^{(\kappa)}p_{k}\right]=0
		\end{equation}
		and
		\begin{align}\label{normal T2} \notag u_{i,j}&=\sum_{1\leq k\leq i-1}\left[\psi_{k}\psi_{j-i+k}^{(\kappa)}-p_{k}p_{j-i+k}^{(\kappa)}\right] +\sum_{1\leq k\leq j-i-1}\left[p_{k}^{(\kappa)}\psi_{j-i-k}^{(\kappa)}-\overline{\psi_k}p_{j-i-k}^{(\kappa)}\right]	\\
			&+ 	\sum_{1\leq k\leq n-j}\left[p_{j-i+k}^{(\kappa)}p_{k}-\psi_{j-i+k}^{(\kappa)}\psi_{k}\right]=0
		\end{align}
		respectively. 
		Let us first evaluate equation (\ref{normal T}) solely on the case $n=2m$, with $m$ denoting a fixed positive integer. In doing so, we proceed to compute the diagonal entries of $U$ as
		\begin{equation}\label{normal Td}	
			u_{r,r} =\sum_{1\leq k \leq r-1}q_k-\sum_{1\leq k\leq 2m-r}q_k= -\left[ \sum_{1\leq k\leq (2m-r+1)-1}q_k-\sum_{1\leq k\leq 2m-(2m-r+1)}q_k \right] =-u_{2m-r+1,2m-r+1},
		\end{equation}
		for every $r$, where  $q_k=\psi_{k}\psi^{(\kappa)}_{k}-p_{k}^{(\kappa)}p_{k}.$ We have $u_{r,r}=-u_{2m-r+1,2m-r+1}$ for all $r$. Thus it is sufficient to focus on the entries $(r,r )$ of $U$,  $r=1,2,\cdots,m$. A simple and straightforward computation, based on 
		(\ref{normal Td}) shows that
		\begin{align*}u_{m,m}&=\sum_{1\leq k\leq m-1}q_k-\sum_{1\leq k\leq m}q_k=-q_m=\psi_{m}\psi^{(\kappa)}_{m}-p_{m}^{(\kappa)}p_{m}=0.
		\end{align*}
		By utilizing an iterative approach, for every $r=1,2,\cdots, m$, one can derive that
		\begin{align*}
			u_{r,r} &=\sum_{1\leq k\leq r-1}q_k-\sum_{1 \leq k\leq 2m-r}q_k= \sum_{1\leq k \leq (r+1)-1}q_k-\sum_{1 \leq k\leq 2m-(r+1)}q_k-(q_{r}+q_{2m-r})\\
			&= u_{r+1,r+1}-( q_{r}+q_{2m-r}).
		\end{align*}
		This leads to $q_{r}+q_{2m-r}=0$, for all $r=1,2,\cdots,m$. As a result, we arrive at 
		\begin{equation}\label{normal Td1}
			p_{r}^{(\kappa)}p_{r}+p_{2m-r}^{(\kappa)}p_{2m-r}=\psi_{r}\psi_{r}^{(\kappa)}+\psi_{2m-r}\psi_{2m-r}^{(\kappa)}.	
		\end{equation}
		We now focus on case $i<j$, rewriting (\ref{normal T2}) as follows
		\[u_{i,j} = \sum_{1\leq k\leq i-1}\mu_{t,k}+\sum_{1\leq k\leq t-1}\nu_{t,k}+\sum_{1\leq k\leq 2m-j}\omega_{t,k},\]
		with $1\leq t = j-i\leq n-1,$
		$\mu_{t,k}= \psi_{k}\psi_{t+k}^{(\kappa)}-p_{k}p_{t+k}^{(\kappa)}$, $\nu_{t,k}=  p_{k}^{(\kappa)}\psi_{t-k}^{(\kappa)}-\psi_k^{(\kappa)}p_{t-k}^{(\kappa)}$, and $\omega_{t,k}=  p_{t+k}^{(\kappa)}p_{k}-\psi_{r+k}^{(\kappa)}\psi_{k}.$
		Since entries are belonging to $\HH$, then
		\begin{eqnarray*}
			u_{i,i+1}& =& \sum_{1\leq k\leq i-1}\mu_{1,k}+\sum_{1\leq k\leq 2m-i-1}\omega_{1,k}	\\
			&= &  \sum_{1\leq k\leq i}\mu_{1,k}+\sum_{1\leq k\leq 2m-i-2}\omega_{1,k}-	\mu_{1,i}+\omega_{1,2m-i-1}\\
			&= &   u_{i+1,i+2}-	\mu_{1,i}+\omega_{1,2m-i-1}.
		\end{eqnarray*}
		Then $	\mu_{1,i}-\omega_{1,2m-i-1}=0$, therefore
		\begin{equation}\label{1} \psi_{i}\psi_{i+1}^{(\kappa)}+\psi_{2m-i}^{(\kappa)}\psi_{2m-(i+1)}=p_{i}p_{1+i}^{(\kappa)}+p_{2m-i}^{(\kappa)}p_{2m-(i+1)},
		\end{equation}
		and
		\begin{eqnarray*}
			u_{i,i+2}& =& \sum_{1\leq k\leq i-1}\mu_{2,k}+\nu_{2,1}+\sum_{1\leq k\leq 2m-i-2}\omega_{2,k}	\\
			&= &  \sum_{1\leq k\leq i}\mu_{2,k}+\nu_{2,1}+\sum_{1\leq k\leq 2m-i-3}\omega_{2,k}-	\mu_{2,i}+\omega_{2,2m-i-2}\\
			&= &   u_{i+1,i+3}-	\mu_{2,i}+\omega_{2,2m-i-2}.
		\end{eqnarray*}
		Then $	\mu_{2,i}-\omega_{2,2m-i-2}=0$,  therefore
		\begin{equation}\label{2} \psi_{i}\psi_{i+2}^{(\kappa)}+\psi_{2m-i}^{(\kappa)}\psi_{2m-(i+2)}=p_{i}p_{2+i}^{(\kappa)}+p_{2m-i}^{(\kappa)}p_{2m-(i+2)}.
		\end{equation}
		By performing similar computations for $u_{i,i+t}$, $t=3,4, \cdots, 2m-i$, we arrive at 
		\begin{equation}\label{3} \psi_{i}\psi_{i+t}^{(\kappa)}+\psi_{2m-i}^{(\kappa)}\psi_{2m-(i+t)}=p_{i}p_{i+t}^{(\kappa)}+p_{2m-i}^{(\kappa)}p_{2m-(i+t)}.
		\end{equation}
		Therefore, based on the equations from 
		(\ref{normal Td})---(\ref{3}), we can infer that if $T^{\dagger \kappa}T-TT^{\dagger \kappa}=0$ then it follows for all $1\leq s,k\leq n-1,$
		\begin{equation}\label{fin}
			\psi_{s}\psi^{(\kappa)}_{k}+\psi^{(\kappa)}_{n-s}\psi_{n-k}=p_s p_k^{(\kappa)}+p_{n-s}^{(\kappa)}p_{n-k}.
		\end{equation}
		For the establishment of converse, let us assume that the condition in equation (\ref{condition}) is true for every pair of indices  $1\leq s,k\leq n-1,$. Our goal is to show that $u_{r,r} = 0$ whenever $1\leq r\leq n $  and $u_{i,j }= 0$ whenever $1\leq i<j\leq n$ respectively. From the argument presented in the first part of the proof, we have established that $u_{r,r} = u_{r+1,r+1} $, for each $r = 1, 2,\cdots, m,$ and that  $u_{r,r}=-u_{2m-r+1,2m-r+1}$ holds for all values of $r$, Additionally, since $u_{m,m}=-q_m=\psi_{m}\psi^{(\kappa)}_{m}-p_{m}^{(\kappa)}p_{m}=0$, it follows that $u_{r,r} = 0$, for every $r = 1, 2,\cdots, 2m$. Thus, the diagonal entries $u_{r,r}$ are all vanish for all $r$.\\\
		For the case where $i < j$, we observe that
		$u_{i,i+1}= u_{i+1,i+2}$, $u_{i,i+2} = u_{i+1,i+3},\cdots ,u_{i,i+t} = u_{i+1,i+t+1}$. These equalities show that $U \in \TT_n[\HH]$. Since entries of $U$ are in $\HH$, we can express $u_{m,m+1}$ as
		$$u_{m,m+1} =\sum_{1\leq k\leq m-1}\mu_{1,k}+\sum_{1\leq k\leq 2m-m-1}\omega_{1,k}=0. $$
		As a result, it follows that $u_{i,i+1} =0$ for all values of $i$.  Furthermore, since
		\begin{eqnarray*}
			u_{m,m+2} = \sum_{1\leq k\leq m-1}\mu_{2,k}+\nu_{2,1}+\sum_{1\leq k\leq 2m-m-2}\omega_{2,k}=0.\\
		\end{eqnarray*}
		Consequently $u_{i,i+2} =0$, for all $i$. 	
		For all $t = 1, 2,\cdots, 2m-i$, we have $u_{i,i+t} = 0,$ A similar
		approach applies for $n =2m + 1$. Therefore, the proof is concluded. 
	\end{proof}

	\section{Acknowledgments}



\begin{thebibliography}{hh}
		
	
	\bibitem{Gu2003} C. Gu, \& Patton (2003). Commutation relation for Toeplitz and Hankel matrices. SIAM J. Matrix Anal. Appl., 24, 728-746.
	
	\bibitem{Zhang1997} F. Zhang (1997). Quaternions and Matrices of Quaternions. Linear algebra and its applications, 251, 21-57.
	

\bibitem{Ye2016} K. Ye, \& L.H. Lim (2016). Every Matrix is a Product of Toeplitz Matrices. Found Comput Math., 16, 577-598.


\bibitem{Heinig2001} G. Heinig (2001). Not every matrix is similar to a Toeplitz matrix. Linear Algebra and its applications., 96, 519-531.

\bibitem{Mackey1999} 
D. S. Mackey, \& S. Petrovic (1999). Is every matrix similar to a Toeplitz matrix?. Linear Algebra and its applications., 297, 87-105.
\bibitem{Williams2020} H. Williams (2020). Superpositions of unitary operators in quantum mechanics. IOPSciNotes, 1, 035204.

\bibitem{Hamilton1969} W. R. Hamilton (1969). Elements of Quaternions. Chelsea Pub. Com..


\bibitem{Hamilton1953} W. R. Hamilton (1953). Lectures on Quaternions. Hodges and Smith.

\bibitem{kleyn2014} Aleks Kleyn. (2014). Lectures on Linear Algebra over Division Ring.


\bibitem{Grenander1958} U. Grenander, \& G. Szego (1958). Toeplitz Forms and Their Applications. University of Calif. Press.

\bibitem{Segre1892} C. Segre (1892). The Real Representations of Complex Elements and Extension to Bicomplex, Systems. Math. Ann., 40, 413-467.

\bibitem{Kosal2014} H. Kosal, \& M. Tosun (2014). Commutative quaternion matrices. Cliiford Algebra, 24, 769-779.

\bibitem{Kosal2015} Kosal, H., Akyigit, M., \& Tosun, M. (2015). Consimilarity of commutative quaternion matrices. Miskolc Math. Notes, 16, 965-977. 
\bibitem{Macklin1984} P. A. Macklin (1984). Normal matrices for physicists. Am. J. Phys., 52, 513-515.

\bibitem{Visick2000} G. Visick (2000). A Quantitative Version of Observation That The Hadamard Product is  a Principal Submatrix of The Kronecker Product.. Linear Algebra and its Applications, 304, 45-68.


\bibitem{Moenck1977OnCC} Robert T. Moenck (1977). On computing closed forms for summations. [polynomials and rational functions].

\bibitem{Mathias1990} R. Mathias (1990). The Spectral Norm of a Nonnegative Matrix. Linear Algebra and its Applications, 131, 269-284.

\bibitem{Altun2021} D. Altun, \& S. Yuce (2021). Algebraic structure and basics of analysis of $n$-dimensional quaternionic space. Heliyon, 7, 1-6.


\bibitem{Gongopadhyay2012} K. Gongopadhyay. (2012). Algebraic Characterization of Isometries of the Hyperbolic 4-Space.

\bibitem{Nikolski2020} N. Nikolski (2020). Toeplitz matrices and operators. Cambridge university press.
\bibitem{Widom1965} H. Widom (1965). Toeplitz Matrices, in Studies in Real and Complex Analysis, , MAA Studies in Mathematics. (J. I. I. Hirschmann, ed.).


\bibitem{Shalom1987} T. Shalom (1987). On algebras of Toeplitz matrices. Linear Algebra and its applications, 24, 211-226.

\bibitem{brown1964} Brown, L., \& Halmos, P. (1964). Algebraic properties of Toeplitz operators. Journal fur die Reine und Angewandte Mathematik, 213, 89–102.


\bibitem{Matej2022} Matej Bresar, \& V. Shulman (2022). On around and beyond Frobenius theorem on division algebra. Linear and multilinear algebra, 1369-1381.

\bibitem{Khan2018} M. A. Khan (2018). A family of maximal algebra of block Toeplitz matrices . Ovidius constanta seria Mathematica, 26, 127-142.



\bibitem{Khan2023} M. A. Khan (2023). Block Toeplitz matrices: multiplicative properties. MATEMATIKA, 24, 101-112.

\bibitem{Khan2021} M. A. Khan, \& D. Timotin (2021). Algebras of block Toeplitz matrices with commuting entries. Linear and Multilinear Algebra, 69, 2702-2716.

\bibitem{Khan20221} M. A. Khan (2022). Product of matrix valued truncated Toplitz 0perators, . Hacet.J. Math. Stat , 51, 700-711.

\bibitem{Khan2022} M. A. Khan, \& A. Yagoub (2022). On some algebraic properties of block Toeplitz matrices with commuting, entries. Operators and Matrices , 16, 909-923.

\bibitem{yagoub2024algebrastoeplitzmatricesquaternion} M. A. Khan, \& A. Yagoub (2024). Algebras of Toeplitz Matrices with Quaternion Entries. Journal of Mathematical Extension , 18, 1-19.

\bibitem{Catoni2005} Catoni, F., Cannata, R., \& Zampetti, P. (2005). An Introduction to Commutative Quaternions. Adv. Appl. Clifford Algebras , 16, 1-28.

\bibitem{Catoni12005} Catoni, F., Cannata, R., Nichelatti, \& Zampetti, P. (2005). Hypercomplex Numbers and Functions of Hypercomplex Variables: a Matrix Study. Adv. Appl. Clifford Algebras , 16, 183-213.
\bibitem{alpay} D. Alpay, F. Colombo, \& I. Sabadini, (2020) Quaternionic de Branges Spaces and Characteristic Operator Function, Springer Cham Switzerland. 
\bibitem{rodman} L. Rodman, Topics in Quaternion Linear Algebra , Princeton University Press, 2014

	\end{thebibliography}
\end{document}